\documentclass[11pt]{amsart}
\usepackage[colorlinks,citecolor=red,dvipdfm]{hyperref}
\usepackage[small,nohug,heads=vee]{diagrams}
\diagramstyle[labelstyle=\scriptstyle]

\newtheorem{theorem}{Theorem}[section]
\newtheorem{lemma}[theorem]{Lemma}

\theoremstyle{definition}
\newtheorem{definition}[theorem]{Definition}

\newtheorem{example}[theorem]{Example}

\newtheorem{proposition}[theorem]{Proposition}
\newtheorem{corollary}[theorem]{Corollary}
\newtheorem{remark}[theorem]{Remark}
\theoremstyle{remark}

\newcommand{\be}{\begin{equation}}
\newcommand{\ee}{\end{equation}}

\numberwithin{equation}{section}

\begin{document}

\title{$-1$-Phenomena for the pluri $\chi_y$-genus and elliptic genus}

\author{Ping Li}
\address{Department of Mathematics, Tongji University, Shanghai 200092, China}
\email{pingli@tongji.edu.cn,\qquad pinglimath@gmail.com}
\thanks{The author was partially supported by the National Natural Science Foundation of China (Grant No.
11101308) and the Fundamental Research Funds for the Central
Universities.}





\begin{abstract}

Several independent articles have observed that the Hirzebruch
$\chi_y$-genus has an important feature, which the author calls
$-1$-phenomenon and tells us that the coefficients of the Taylor
expansion of the $\chi_y$-genus at $y=-1$ have explicit expressions.
Hirzebruch's original $\chi_y$-genus can be extended towards two
directions: the pluri-case and the case of elliptic genus. This
paper contains two parts in which  we investigate the $-1$-phenomena
in these two generalized cases respectively and show that in each
case there exists a $-1$-phenomenon in a suitable sense. Our main
results in the first part have an application, which states that all
characteristic numbers (Chern numbers and Pontrjagin numbers) on
manifolds can be expressed, in a very explicit way, in terms of some
rationally linear combination of indices of some elliptic operators.
This gives an analytic interpretation of characteristic numbers and
affirmatively answers a question posed by the author several years
ago. The second part contains our attempt to generalize this
$-1$-phenomenon to elliptic genus, a modern version of the
$\chi_y$-genus. We first extend the elliptic genus of an
almost-complex manifold to a twisted version where an extra complex
vector bundle is involved, and show that it is a weak Jacobi form
under some assumptions. A suitable manipulation on the theory of
Jacobi form will produce new modular forms from this weak Jacobi
form and thus much arithmetic information related to the underlying
manifold can be obtained, in which the $-1$-phenomenon of the
original $\chi_y$-genus is hidden.
\end{abstract}

\maketitle

\tableofcontents

\section{Introduction}

\subsection{The Hirzebruch $\chi_y$-genus and its
$-1$-phenomenon}\label{section1.2}

 In his highly influential book
\cite{Hi}, Hirzebruch defined a polynomial with integral
coefficients $\chi_y(M)$ for projective manifolds $M$, which encodes
the information of indices of Dolbeault complexes and is now called
the \emph{Hirzebruch $\chi_y$-genus}. After the discovery of the
general index theorem due to Aityah and Singer, we know that
$\chi_y(\cdot)$ can be defined on compact almost-complex manifolds
and computed in terms of Chern numbers as follows.

Suppose $(M^{2d},J)$ is a compact connected almost-complex manifold
with an almost-complex structure $J$. The choice of an almost
Hermitian metric on $M$ enables us to define the Hodge star operator
$\ast$ and the formal adjoint
$\bar{\partial}^{\ast}=-\ast\bar{\partial}~\ast$ of the
$\bar{\partial}$-operator. For each pair $0\leq p,q\leq d$, we
denote by
$$\Omega^{p,q}(M):=\Gamma(\Lambda^{p}T^{\ast}M\otimes\Lambda^{q}
\overline{T^{\ast}M})$$
 the complex vector space which consists of
smooth complex-valued $(p,q)$-forms. Here $T^{\ast}M$ is the dual of
holomorphic tangent bundle $TM$ in the sense of $J$. Then for each
$0\leq p\leq d,$ we have the following Dolbeault-type elliptic
differential operator \be\bigoplus_{q~\textrm{even}}\Omega^{p,q}(M)
\xrightarrow{(\bar{\partial}+ \bar{\partial}^{\ast})|_p}\bigoplus_{q
~\textrm{odd}}\Omega^{p,q}(M),\nonumber\ee whose index is denoted by
$\chi^{p}(M)$ in the notation of Hirzebruch \cite{Hi}. Then the
Hirzebruch $\chi_{y}$-genus of $M$ is nothing but the generating
function of these indices $\chi^p(M)$ $(0\leq p\leq d):$
$$\chi_{y}(M):=\sum_{p=0}^{d}\chi^{p}(M)\cdot y^{p}.$$
If we denote by $x_ 1,\ldots,x_d$ the formal Chern roots of $TM$,
i.e., the $i$-th elementary symmetric polynomial of $x_1,\ldots,x_d$
represents $c_i$, the $i$-th Chern class of $TM$, then the general
form of the Hirzebruch-Riemann-Roch theorem (first proved by
Hirzebruch for projective manifolds \cite{Hi}, and in the general
case by Atiyah and Singer \cite{AS}) tells us that
\be\label{HRR}\chi_{y}(M)=\int_M\prod_{i=1}^d\frac{x_i(1+y
  e^{-x_i})}{1-e^{-x_i}}.\ee
Among other things, the Hirzebruch $\chi_y$-genus has an important
feature, which the author calls \emph{``$-1$-phenomenon"} and has
been noticed, implicitly or explicitly, in several independent
articles (\cite{NR}, \cite{LW}, \cite{Sa}). This $-1$-phenomenon
says that,  at $y=-1$, the coefficients of the Taylor expansion of
$\chi_y(M)$ have explicit expressions. To be more precise, if we
write \be\label{-1}\chi_y(M)=:\sum_{i=0}^d a_i(M)\cdot(y+1)^i,\ee
 then these $a_i(M)$ can
be given explicit expressions in terms of Chern numbers of
$(M^{2d},J)$ as follows. \be\label{expressions}
\begin{split}
&a_{0}(M)=c_{d},\qquad a_{1}(M)= -\frac{1}{2}dc_{d},\qquad a_{2}(M)=
\frac{1}{12}[\frac{d(3d-5)}{2}c_{d}+c_{1}c_{d-1}],\\
&a_{3}(M)=-\frac{1}{24}[\frac{d(d-2)(d-3)}{2}c_{d}+(d-2)c_{1}c_{d-1}],\qquad\cdots.
\end{split}\ee

 By definition, these $a_i(M)$ are integers. Thus immediate
consequences of their expressions are divisibility properties of
Chern numbers. The derivation of these expressions is direct. That
is, to expand the right-hand side of (\ref{HRR}) at $y=-1$ and
express the coefficients in terms of elementary symmetric
polynomials of $x_1,\ldots,x_d$. The calculations of $a_0$ and $a_1$
are quite easy. The calculation of $a_2$ appears implicitly in
\cite[p. 18]{NR} and explicitly in \cite[p. 141-143]{LW}. Narasimhan
and Ramanan used $a_2$ to give a topological restriction on some
moduli spaces of stable vector bundles on smooth projective
varieties. Libgober and Wood used $a_2$ to prove the uniqueness of
the complex structure on K\"{a}hler manifolds of certain homotopy
types. Inspired by \cite{NR}, S. Salamon applied $a_2$
\cite[Corollary 3.4]{Sa} to obtain a restriction on the Betti
numbers of hyper-K\"{a}hler manifolds \cite[Theorem 4.1]{Sa}. The
expressions of $a_3$ and $a_4$ are also included in \cite[p.
145]{Sa}. Hirzebruch used $a_1$, $a_2$ and $a_3$ to obtain a
divisibility result on the Euler characteristic of those
almost-complex manifolds whose $c_1c_{d-1}=0$ (\cite{Hi2}). In
particular, those almost-complex manifolds with $c_1=0$ satisfy this
property.

\subsection{Pluri-$\chi_y$-genus}\label{section1.3}
Some acquaintance with index theory will lead to the observation
that $\chi_y(M)$ is the index of the following Todd operator whose
index is the Todd genus
\be\label{dolbeault}{\Omega^{0,\textrm{even}}(M)
\xrightarrow{(\bar{\partial}+
\bar{\partial}^{\ast})|_0}\Omega^{0,\textrm{odd}}(M)}\ee twisted by
$\Omega_y(M)$, where
$$\Omega_y(M):=\sum_{p=0}^{d}\Lambda^p(T^{\ast}M)\cdot y^p\in
K(M)[y]$$ and $\Lambda^p(\cdot)$ \big(resp. $K(\cdot)$\big) denotes
the $p$-th exterior power (resp. K-group). Therefore $\chi_y(M)$ can
be rewritten as follows:
$$\chi_y(M)=\text{Ind}\big((\bar{\partial}+
\bar{\partial}^{\ast})|_0\otimes\Omega_y(M)\big)=:\chi\big(M,\Omega_y(M)\big).$$

Here for simplicity we denote by the standard notation
$\chi\big(M,(\cdot)\big)$ the index of the Todd operator
(\ref{dolbeault}) twisted by an element $(\cdot)\in K(M)$.

 We can also consider, for arbitrarily fixed positive integer
$g$, the \emph{pluri $\chi_y$-genus} $\chi_{\underline{y}}(M)$ by
using sufficiently many forms of the type \be
\begin{split}\Omega_{\underline{y}}(M):&= \sum_{0\leq
p_1,\ldots,p_g\leq d}\Lambda^{p_{1}}(T^{\ast}M) \otimes \dots
\otimes \Lambda^{p_{g}}(T^{\ast}M)\cdot y_{1}^{p_{1}}\cdots
y_{g}^{p_{g}}\\
&=\Omega_{y_1}(M)\otimes\cdots\otimes\Omega_{y_g}(M)\in
K(M)[y_1,\ldots,y_g]
\end{split}\nonumber \ee
to twist $(\bar{\partial}+ \bar{\partial}^{\ast})|_0$, i.e.,
$$\chi_{\underline{y}}(M):=\text{Ind}\big((\bar{\partial}+
\bar{\partial}^{\ast})|_0
\otimes\Omega_{\underline{y}}(M)\big)=\chi\big(M,\Omega_{\underline{y}}(M)\big),$$
which specializes to the Hirzebruch's original $\chi_y$-genus when
$g=1$.

Inspired by the above-mentioned $-1$-phenomenon of the
$\chi_y$-genus, we may ask what the coefficients look like if we
expand $\chi_{\underline{y}}(M)$ at $y_1=\cdots=y_g=-1$. Our first
main observation in this article is that the coefficients of
$(y+1)^{p_1}\cdots(y+1)^{p_g}$ in $\chi_{\underline{y}}(M)$ can be
divided into three parts, which is our main result in Section
\ref{section3} (Theorem \ref{mainresult3}). Moreover, we can do a
similar manipulation for signature operator on closed smooth
oriented manifolds and their coefficients also have a similar
feature (Theorem \ref{mainresult4}). A direct corollary of these two
theorems is that any Chern number of $(M^{2d},J)$ (resp. any
Pontrjagin number of a closed smooth oriented manifold) can be
written as a rationally linear combination of indices of some
elliptic operators explicitly, which provides an analytic
interpretation of characteristic numbers and answers a question of
the author proposed in \cite[Question 1.1]{Li} affirmatively.

\subsection{Elliptic genus}\label{section1.4}
Elliptic genera of oriented differentiable manifolds and
almost-complex manifolds were first constructed by Ochanine,
Landweber-Stong and Hirzebruch in a topological way and Witten gave
it a geometric interpretation, which can be viewed as loop spaces'
analogues to the Hirzebruch $L$-genus and $\chi_y$-genus (cf.
\cite{La} and the references therein). The most remarkable property
of elliptic genera is their rigidity for spin manifolds and
almost-complex Calabi-Yau manifolds \big(in the very weak sense that
$c_1$ vanishes up to torsion, i.e., $c_1=0\in
H^2(M,\mathbb{R})$\big), which was conjectured by Witten and
generalizes the famous rigidity property of the original $L$-genus,
$\hat{A}$-genus (\cite{AH}) and $\chi_y$-genus (\cite{Lu}). The
first rigorous proof was presented by Bott and Taubes (\cite{Ta},
\cite{BT}). A quite simple, unified and enlightening proof was
discovered by Liu (\cite{Liu1}), in which modular invariance of the
four classical Jacobi-theta functions and their various
transformation laws play key roles. Later on, this modular
invariance property, its variously remarkable extensions and
relation with vertex operator algebra were established by Liu and
his coauthors from various aspects (\cite{Liu}, \cite{Liu2},
\cite{LiuMa}, \cite{LiuMaZhang1}, \cite{LiuMaZhang2}, \cite{DLM},
\cite{CHZ}, \cite{CH}, \cite{HZ}, \cite{HLZ}, \cite{HL} etc.).

 What we are concerned
with in this paper is the elliptic genus of almost-complex
manifolds. The elliptic genus of a compact, almost-complex manifold
$(M^{2d},J)$, which we denote by $\textrm{Ell}(M,\tau,z)$, is
defined as a function of two variables
$(\tau,z)\in\mathbb{H}\times\mathbb{C}$, where $\mathbb{H}$ is the
upper half plane. To be more precise, $\textrm{Ell}(M,\tau,z)$ is
defined to be the index of the Todd operator (\ref{dolbeault})
twisted by
$$y^{-\frac{d}{2}}\otimes_{n\geq 1}(\Lambda_{-yq^{n-1}}T^{\ast}
\otimes\Lambda_{-y^{-1}q^n}T\otimes\textrm{S}_{q^n}T^{\ast}
\otimes\textrm{S}_{q^n}T)=:\textrm{E}_{q,y},$$ i.e.,
$\textrm{Ell}(M,\tau,z):=\chi(M,\textrm{E}_{q,y})$, where
$q=e^{2\pi\sqrt{-1}\tau}$, $y=e^{2\pi\sqrt{-1}z}$ and $T$ (resp.
$T^{\ast}$) is the holomorphic (resp. dual of holomorphic) tangent
bundle of $M$ in the sense of $J$. Here
$$\Lambda_t(W):=\bigoplus_{i\geq 0}\Lambda^i(W)\qquad \textrm{and}
\qquad S_t(W):=\bigoplus_{i\geq 0}S^i(W)$$ for any complex vector
bundle $W$ denote the generating series of the exterior and
symmetric powers of $W$ respectively.

According to the Atiyah-Singer index theorem we have
\be\begin{split}
\textrm{Ell}(M,\tau,z)&=\int_{M}\textrm{td}(M)\cdot\textrm{ch}
(\textrm{E}_{q,y})\\
&=y^{-\frac{d}{2}}\chi_{-y}(M)+q\cdot\big[y^{-\frac{d}{2}}\chi_{-y}
\big(M,T^{\ast}(1-y)+T(1-y^{-1})\big)\big]+q^2\cdot(\cdots),\end{split}\nonumber\ee
where $$\textrm{td}(M):=\prod_{i=1}^d\frac{x_i}{1-e^{-x_i}}$$ is the
Todd class of $M$ and $\textrm{ch}(\cdot)$ is the Chern character.

Thus elliptic genus $\textrm{Ell}(M,\tau,z)$ can be viewed as a
generalization of the Hirzebruch $\chi_y$-genus in the sense that
the $q^0$-term of the Fourier expansion of $\textrm{Ell}(M,\tau,z)$
is essentially $\chi_y(M)$. If $(M^{2d},J)$ is Calabi-Yau, the
coefficients of $q$-expansion of $\textrm{Ell}(M,\tau,z)$ are rigid
for arbitrary $y$ (\cite[Theorem B]{Liu1}). Moreover, in this case,
$\textrm{Ell}(M,\tau,z)$ itself is a weak Jacobi form of weight $0$
and index $\frac{d}{2}$ (\cite[Theorem 2.2]{BL}, \cite[Proposition
1.2]{Gr1}).

As we have mentioned above, elliptic genus $\textrm{Ell}(M,\tau,z)$
can be viewed as a generalization of $\chi_y(M)$ and also has a
rigidity property when $M$ is Calabi-Yau. So we may ask, in the case
of Calabi-Yau, whether $\textrm{Ell}(M,\tau,z)$ has some kind of
arithmetic phenomenon which extends the original $-1$-phenomenon of
$\chi_y(M)$. Note that, strictly speaking, $\textrm{Ell}(M,\tau,z)$
is a generalization of $\chi_{-y}(M)$ rather than $\chi_y(M)$ as the
$q^0$-term of $\textrm{Ell}(M,\tau,z)$ is
$y^{-\frac{d}{2}}\chi_{-y}(M)$. So if there exists some kind of
phenomenon which extends the original $-1$-phenomenon of
$\chi_y(M)$, the parameter $y=e^{2\pi\sqrt{-1}z}$ should correspond
to $1$ rather than $-1$. Thus the variable $z$ should correspond to
$0$. Indeed, there does exist such a kind of generalization, which
depends on some arithmetic properties of Jacobi form and has been
implicitly used by Gritsenko in \cite{Gr1}. Our aim in Section
\ref{section3} is two-fold. On the one hand, given a compact
almost-complex manifold $(M^{2d},J)$ and a rank $l$ complex vector
bundle $W$ over it, we construct a generalized elliptic genus
$\textrm{Ell}(M,W,\tau,z)$, which is defined to be the index of the
Todd operator (\ref{dolbeault}) twisted by
$$[\prod_{i=1}^{\infty}(1-q^i)]^{2(d-l)}\cdot y^{-\frac{l}{2}}\otimes_{n\geq 1}(\Lambda_{-yq^{n-1}}W^{\ast}
\otimes\Lambda_{-y^{-1}q^n}W\otimes\textrm{S}_{q^n}T^{\ast}
\otimes\textrm{S}_{q^n}T),$$ and show that it is a weak Jacobi form
of weight $d-l$ and index $\frac{l}{2}$ if the first Pontrjagin
classes $p_1(M)=p_1(W)$ and the first Chern class $c_1(W)=0$ in
$H^{\ast}(M,\mathbb{R})$. On the other hand, we highlight a
well-known manipulation in Jacobi form to obtain modular forms from
$\textrm{Ell}(M,W,\tau,z)$, whose arithmetic information will in
turn give geometric results on $M$ and $W$. Some examples are given
to illustrate this observation.

\section*{Acknowledgements}
The first part of this paper was inspired by my fruitful discussions
with George Thompson via emails in 2010 and 2011. The second part of
this paper was suggested to me by Fei Han when I visited him at the
National University of Singapore in 2012. This paper was initiated
when I was holding a JSPS Postdoctoral Fellowship for Foreign
Researchers at Waseda University in Japan with the help of Martin
Guest. To all of them I would like to express my sincere thanks.

\section{$-1$-phenomenon of the pluri-$\chi_y$-genus}\label{section2}
\subsection{Statements of the main results related to the pluri-$\chi_y$-genus}
Let $(M^{2n},J)$ (resp. $X^{2n}$) be a compact almost-complex
manifold of complex dimension $n$ (resp. smooth, closed oriented
manifold of real dimension $2n$). As before we use $(\bar{\partial}+
\bar{\partial}^{\ast})|_0$ to denote the Todd operator on
$(M^{2n},J)$ whose index is the Todd genus of $M$. We denote by $D$
the signature operator on $X$, whose index is the signature of
$X^{2n}$ (\cite[\S 6]{AS}). By definition $\text{Ind}(D)$ is zero
unless $n$ is even.

Let $W$ be a complex vector bundle over $M$ (resp. $X$). By means of
a connection on $W$, the elliptic operator $(\bar{\partial}+
\bar{\partial}^{\ast})|_0$ and $D$ can be extended to a new elliptic
operator $\big((\bar{\partial}+
\bar{\partial}^{\ast})|_0\big)\otimes W$ and $D\otimes W$, whose
indices via the Atiyah-Singer index theorem are
\be\begin{split}\chi(M,W)=\textrm{Ind}\big(\big((\bar{\partial}+
\bar{\partial}^{\ast})|_0\big)\otimes
W\big)&=\int_M[\text{td}(M)\cdot\textrm{ch}(W)]\\
&=\int_M[\prod_{i=1}^{n}
\frac{x_i}{1-e^{-x_i}}\cdot\textrm{ch}(W)]\end{split}\nonumber\ee
and
$$\text{Ind}\big((D\otimes W)\big)=\int_X
[(\prod_{i=1}^{n}\frac{x_i}{\tanh{\frac{x_i}{2}}})\cdot\text{ch}(W)]$$
respectively. Here we use the $i$-th elementary symmetry polynomial
of $x_1,\ldots,x_n$ (resp. $x_1^2,\ldots,x_n^2$) to denote the
$i$-th Chern class (resp. Pontrjagin class) of $(M^{2n},J)$ (resp.
$X^{2n}$).
\begin{definition}
For arbitrarily fixed positive integer $g$, we define
 \be\begin{split}\Omega_{\underline{y}}(M):&= \sum_{0\leq
p_1,\ldots,p_g\leq n}\Lambda^{p_{1}}(T^{\ast}M) \otimes \dots
\otimes \Lambda^{p_{g}}(T^{\ast}M)\cdot y_{1}^{p_{1}}\cdots
y_{g}^{p_{g}}\\
&=\Omega_{y_1}(M)\otimes\cdots\otimes\Omega_{y_g}(M)\in
K(M)[y_1,\ldots,y_g],
\end{split}\nonumber \ee

\be\begin{split}\Omega_{\underline{y}}^\mathbb{R}(X):&= \sum_{0\leq
p_1,\ldots,p_g\leq 2n}\Lambda^{p_{1}}(T^{\ast}_{\mathbb{C}}X)
\otimes \dots \otimes \Lambda^{p_{g}}(T^{\ast}_{\mathbb{C}}X)\cdot
y_{1}^{p_{1}}\cdots
y_{g}^{p_{g}}\\
&=\Omega_{y_1}^{\mathbb{R}}(X)\otimes\cdots\otimes\Omega_{y_g}^{\mathbb{R}}(X)\in
\big(KO(X)\otimes\mathbb{C}\big)[y_1,\ldots,y_g],
\end{split}
\nonumber \ee where
$$\Omega_{y}^{\mathbb{R}}(X):=\sum_{p=0}^{2n}\Lambda^{p}(T^{\ast}_{\mathbb{C}}X)\cdot
y^p$$ and $T^{\ast}_{\mathbb{C}}X$ is the dual of the complexified
tangent bundle of $X$, and
\be\begin{split}\chi_{\underline{y}}(M):&=\sum_{0\leq
p_1,\ldots,p_g\leq n}\textrm{Ind}\big[(\bar{\partial}+
\bar{\partial}^{\ast})|_0\otimes\big(\Lambda^{p_{1}}(T^{\ast}M)
\otimes
\dots \otimes \Lambda^{p_{g}}(T^{\ast}M)\big)\big]\cdot y_{1}^{p_{1}} \cdots y_{g}^{p_{g}}\\
&=\int_M\big[\prod_{i=1}^{n}\frac{x_i}{1-e^{-x_i}}\cdot\textrm{ch}(\Omega_{\underline{y}}(M))\big]
\end{split}\nonumber\ee

\be\begin{split}D_{\underline{y}}(X):&= \sum_{0\leq
p_1,\ldots,p_g\leq
2n}\textrm{Ind}\big[D\otimes\big(\Lambda^{p_{1}}(T^{\ast}_{\mathbb{C}}X)
\otimes \dots \otimes
\Lambda^{p_{g}}(T^{\ast}_{\mathbb{C}}X)\big)\big]\cdot y_{1}^{p_{1}}
\cdots y_{g}^{p_{g}}\\
&=\int_X\big[(\prod_{i=1}^{n}\frac{x_i}
{\tanh{\frac{x_i}{2}}})\cdot\textrm{ch}(\Omega_{\underline{y}}^\mathbb{R}(X))\big].
\end{split}\nonumber\ee
\end{definition}

Our main result in this section is
\begin{theorem}\label{mainresult3}
{\rm{The coefficient of
$(1+y_{1})^{n-q_{1}}\cdots(1+y_{g})^{n-q_{g}}$ in
$\chi_{\underline{y}}(M)$ is equal to
\begin{eqnarray}
\left\{ \begin{array}{ll} 0, & \textrm{if
$\sum_{i=1}^{g}q_{i} > n$},\\
\int_M\prod_{i=1}^gc_{q_i}(M), & \textrm{if $\sum_{i=1}^{g}q_{i}=n$},\\
\text{a rationally linear combination of Chern numbers of $M$}, &
\textrm{if $\sum_{i=1}^{g}q_{i}<n$}.
\end{array} \right.
\nonumber\end{eqnarray}}}
\end{theorem}

We have a similar result for smooth manifolds.

\begin{theorem}\label{mainresult4}
 {\rm{ If $n$ is even, the coefficient of
$$(1+y_{1})^{2(n-q_{1})}\cdots(1+y_{g})^{2(n-q_{g})}$$
 in
$D_{\underline{y}}(X)$ is equal to
\begin{eqnarray}
\left\{ \begin{array}{ll} 0, & \textrm{if
$\sum_{i=1}^{g}q_{i} > \frac{n}{2}$},\\
(-1)^{\frac{n}{2}}\cdot 2^{n}\cdot\int_X\prod_{i=1}^g p_{q_i}(X),
& \textrm{if $\sum_{i=1}^{g}q_{i}=\frac{n}{2}$},\\
\text{a rationally linear combination of Pontrjagin numbers of $X$},
& \textrm{if $\sum_{i=1}^{g}q_{i}<\frac{n}{2}$},
\end{array} \right.
\nonumber\end{eqnarray} where $p_i(X)$ is the $i$-th Pontrjagin
class of $X$.
 }}
\end{theorem}

Clearly a direct corollary of this theorem is the following result,
which gives an affirmative answer to a question proposed by the
author in \cite[Question 1.1]{Li}.

\begin{corollary}
Any Chern number (resp. Pontrjagin number) on a compact
almost-complex manifold (resp. compact smooth manifold) can be
expressed, in a very explicit way, in terms of the indices of some
elliptic differential operators over this manifold.
\end{corollary}

\subsection{Proofs of Theorems \ref{mainresult3} and \ref{mainresult4}}
By the abuse of notation we use $c_q(\cdots)$ to denote both the
$q$-th Chern class of an almost-complex manifold and the $q$-th
elementary symmetric polynomial of the variables in the bracket.

The proofs of Theorems \ref{mainresult3} and \ref{mainresult4}
depend on the following lemma.
\begin{lemma}\label{lemma2}
{\rm{If we assign each $x_i$ ($1\leq i\leq n$) the same degree, then
we have
\begin{enumerate} \item The coefficient of
$(1+y)^{n-q}$ ($0\leq q\leq n$) in $\prod_{i=1}^n(1+ye^{-x_i})$ is
$$c_q(x_1,\ldots,x_n)+\textrm{higher degree terms}.$$

\item The coefficient of
$(1+y)^{2(n-q)}$ ($0\leq q\leq n$) in
$\prod_{i=1}^n(1+ye^{-x_i})(1+ye^{x_i})$ is
$$(-1)^qc_q(x_1^2,\ldots,x_n^2)+\textrm{higher degree terms}.$$
\end{enumerate}}}
\end{lemma}

\begin{proof}
$$\prod_{i=1}^n(1+ye^{-x_i})=\prod_{i=1}^n[(1-e^{-x_i})+
(1+y)e^{-x_i}]=e^{-c_1}\prod_{i=1}^n[(e^{x_i}-1)+(1+y)].$$

Thus the coefficient of $(1+y)^{n-q}$ in
$\prod_{i=1}^n(1+ye^{-x_i})$ is
$$e^{-c_1}\cdot
c_q(e^{x_1}-1,\ldots,e^{x_n}-1)=c_q(x_{1},\ldots,x_{n})+\textrm{higher
degree terms}.$$

Similarly,
$$\prod_{i=1}^n(1+ye^{-x_i})(1+ye^{x_i})=
\prod_{i=1}^n[(e^{x_i}-1)+(1+y)][(e^{-x_i}-1)+(1+y)]$$ and the
coefficient of $(1+y)^{2n-q}$ is \be\begin{split} &
c_q(e^{x_1}-1,\ldots,e^{x_n}-1,e^{-x_1}-1,\ldots,e^{-x_n}-1)\\
=&c_q(x_1,\ldots,x_n,-x_1,\ldots,-x_n)+\textrm{higher degree terms}
\end{split}\nonumber\ee

Note that
\begin{eqnarray}
c_q(x_1,\ldots,x_n,-x_1,\ldots,-x_n) = \left\{ \begin{array}{ll} 0,
& \textrm{if
$q$ is odd},\\
(-1)^{\frac{q}{2}}c_{\frac{q}{2}}(x_1^2,\ldots,x_n^2), & \textrm{if
$q$ is even}.
\end{array} \right.
\nonumber\end{eqnarray}

This gives the desired property.
\end{proof}

Now we can prove Theorems \ref{mainresult3} and \ref{mainresult4}.

\begin{proof}
If we use $x_1,\ldots,x_n$ (resp. $x_1,\ldots,x_n,-x_1,\ldots,-x_n$)
to denote the formal Chern roots of $TM$ (resp. $T_{\mathbb{C}}X$),
 then we have (cf. \cite[p. 11]{HBJ})
$$\textrm{ch}\big(\Omega_{\underline{y}}(M)\big)=
\prod_{j=1}^g\big[\prod_{i=1}^n(1+y_je^{-x_i})\big]$$
and
$$\textrm{ch}\big(\Omega_{\underline{y}}^\mathbb{R}(X)\big)=
\prod_{j=1}^g\big[\prod_{i=1}^n(1+y_je^{-x_i})(1+y_je^{x_i})\big].$$

Thus \be\begin{split}\chi_{\underline{y}}(M)
&=\int_M\big[(\prod_{i=1}^{n}\frac{x_i}{1-e^{-x_i}})\cdot\textrm{ch}(\Omega_{\underline{y}}(M))\big]\\
&=\int_M\bigg\{(\prod_{i=1}^{n}\frac{x_i}{1-e^{-x_i}})
\cdot\prod_{j=1}^g\big[\prod_{i=1}^n(1+y_je^{-x_i})]\bigg\}\end{split}\nonumber\ee
and \be\begin{split}\textrm{ind}(D_{\underline{y}}^{\mathbb{R}}(X))
&=\int_X\big[(\prod_{i=1}^{n}\frac{x_i}
{\text{tanh$\frac{x_i}{2}$}})\cdot\textrm{ch}(\Omega_{\underline{y}}^\mathbb{R}(X))\big]\\
&=\int_X\bigg\{(\prod_{i=1}^{n}\frac{x_i}
{\text{tanh$\frac{x_i}{2}$}}) \cdot\prod_{j=1}^g\big[\prod_{i=1}^n
(1+y_je^{-x_i})(1+y_je^{x_i})\big]\bigg\}.\end{split}\nonumber\ee
\end{proof}

Note that the constant terms of
$$\frac{x_i}{1-e^{-x_i}}=1+\cdots$$and
$$\frac{x_i}
{\text{tanh$\frac{x_i}{2}$}}=\frac{x_i(1+e^{-x_i})}{1-e^{-x_i}}=2+\cdots$$
are $1$ and $2$ respectively. So by Lemma \ref{lemma2}, when
considering the Taylor expansion of
$\textrm{ind}(D_{\underline{y}}(M))$ \big(resp.
$\textrm{ind}(D_{\underline{y}}^{\mathbb{R}}(X))$\big) at
$y_1=\cdots=y_g=-1$, the coefficients before the terms
$(1+y_{1})^{n-q_{1}}\cdots(1+y_{g})^{n-q_{g}}$ \big(resp.
$(1+y_{1})^{2(n-q_{1})}\cdots(1+y_{g})^{2(n-q_{g})}$\big) are
\be\begin{split}&\int_M\big\{(1+\textrm{higher degree
terms})\cdot\prod_{j=1}^g\big[c_{q_j}(x_1,\ldots,x_n)+
\textrm{higher degree terms}\big]\big\}\\
&=\int_M\prod_{i=1}^gc_{q_i}(M)+\int_M(\textrm{higher degree
terms}),
\end{split}\nonumber\ee
 \be\begin{split}\bigg(\text{resp}.~&\int_X\big\{(2^n+\textrm{higher degree
terms})\cdot\prod_
{j=1}^g\big[(-1)^{q_j}c_{q_j}(x_1^2,\ldots,x_n^2)+ \textrm{higher
degree terms}\big]\big\}\\
&=2^n\cdot(-1)^{\sum_{j=1}^gq_j}\int_X\prod_{j=1}^gp_{q_i}(X)+\int_X(\textrm{higher
degree terms}),\bigg)
\end{split}\nonumber\ee
which give the desired proofs of Theorems \ref{mainresult3} and
\ref{mainresult4}.

\section{Generalized elliptic genus and its $-1$-phenomenon}\label{section3}
\subsection{Generalized elliptic genus of almost-complex manifolds}
In this subsection we extend the original definition of elliptic
genus of almost-complex manifolds by considering an extra complex
vector bundle and show that it is a weak Jacobi form. As before, let
$(M^{2d},J)$ be a compact almost-complex manifold and $W$ be a rank
$l$ complex vector bundle over it.

\begin{definition}
The generalized elliptic genus of $(M^{2d},J)$ with respect to $W$,
which we denote by $\textrm{Ell}(M,W,\tau,z)$, is defined to be the
index of the Todd operator
$$\Omega^{0,\textrm{even}}(M)\xrightarrow{(\bar{\partial}+
\bar{\partial}^{\ast})|_{0}}\Omega^{0,\textrm{odd}}(M)$$ twisted by
$$c^{2(d-l)}\cdot
y^{-\frac{l}{2}}\otimes_{n\geq 1}(\Lambda_{-yq^{n-1}}W^{\ast}
\otimes\Lambda_{-y^{-1}q^n}W\otimes\textrm{S}_{q^n}T^{\ast}
\otimes\textrm{S}_{q^n}T)=:\textrm{E}(W,q,y),$$ where
$$q=e^{2\pi\sqrt{-1}\tau},~y=e^{2\pi\sqrt{-1}z},$$
and for simplicity $c:=\prod_{i=1}^{\infty}(1-q^i).$

If $W=T$, this definition degenerates to the original elliptic
genus.
\end{definition}

Our first observation in this section is the following result, which
extends \cite[Theorem 2.2]{BL} and \cite[Proposition 1.2]{Gr1} in
which case $W=T$.

\begin{theorem}\label{theoremofellipticgenus}

{\rm{The generalized elliptic genus $\textrm{Ell}(M,W,\tau,z)$ is a
weak Jacobi form of weight $d-l$ and index $\frac{l}{2}$ provided
that the first Pontrjagin classes $p_1(M)=p_1(W)$ and the first
Chern class $c_1(W)=0$ in $H^{\ast}(M,\mathbb{R})$.}}
\end{theorem}

\begin{remark}
~
\begin{enumerate}
\item
A two-variable function $\varphi(\tau,z)$ for
$(\tau,z)\in\mathbb{H}\times\mathbb{C}$ is called a \emph{weak
Jacobi form of weight $k$ and index $m$} for $k\in\mathbb{Z}$ and
$m\in\mathbb{Z}/2$ if it is a holomorphic function with respect to
the two variables $\tau$ and $z$, has no negative powers of $q$ in
its Fourier expansion in terms of $q^iy^j$ and satisfies some
transformation laws involving $k$ and $m$, where the precise
definition can be found in \cite[p. 104, p. 9]{EZ}. In \cite{EZ} the
only integral indices are considered. However, with some minor
modifications of inserting a character, this notion can be easily
extended to the case where the index is allowed to be a
half-integer. (cf. \cite[p. 102]{Gr1}).

\item
Motivated by his ingenious proof of the rigidity theorem, Liu
constructed a two-variable function for $(M,J)$ and $W$ and showed
that it is a weak Jacobi form under some assumptions and the
original Witten theorem exactly corresponds to the case where the
index is equal to zero (\cite[Theorem 3, Corollary 3.1]{Liu2}). This
construction later was generalized to the family case by Liu-Ma
(\cite[Theorem 3.1]{LiuMa}). So our theorem has a similar flavor to
those of Liu and Ma.

\item
In \cite[Theorem 1.2]{Gr2}, Gritsenko further extended the original
elliptic genus to another case where an extra complex bundle is
involved. But his construction is different from ours as it is still
of weight zero.
\end{enumerate}
\end{remark}

The Atiyah-Singer index theorem tells us that
$$\textrm{Ell}(M,W,\tau,z)=\int_M\textrm{td}(M)\cdot\textrm{ch}\big(\textrm{E}(W,q,y)\big).$$

In particular, if $J$ is integrable, $\textrm{Ell}(M,W,\tau,z)$ is
the holomorphic Euler characteristic of the (virtual) bundle
$\textrm{E}(W,q,y)$.

Let us recall one of the famous Jacobi-theta series (\cite[Chapter
$5$]{Ch}) \be\begin{split}
\theta(\tau,z):&=\sum_{n\in\mathbb{Z}}(-1)^nq^{\frac{(n+\frac{1}{2})^2}{2}}
y^{n+\frac{1}{2}}\\
&=2cq^{\frac{1}{8}}\textrm{sin}(\pi z)\prod_{n=1}^{\infty}
(1-q^ny)(1-q^ny^{-1})\\
&=2cq^{\frac{1}{8}}\textrm{sinh}(\pi\sqrt{-1}z)\prod_{n=1}^{\infty}
(1-q^ny)(1-q^ny^{-1})\\
&=2cq^{\frac{1}{8}}\textrm{sinh}(\pi\sqrt{-1}z)\prod_{n=1}^{\infty}
(1-q^ne^{2\pi\sqrt{-1}z})(1-q^ne^{-2\pi\sqrt{-1}z}).
\end{split}\nonumber\ee

The following lemma says that $\textrm{Ell}(M,W,\tau,z)$ can be
expressed in terms of $\theta(\tau,z)$.

\begin{lemma}
{\rm{If we denote by $2\pi\sqrt{-1}x_i$ ($1\leq i\leq d$) and
$2\pi\sqrt{-1}w_i$ ($1\leq i\leq l$) respectively the Chern roots of
$\textrm{T}M$ and $W$, then we have
\be\begin{split}&\textrm{Ell}(M,W,\tau,z)\\
=&\int_M\bigg[\textrm{exp}\big(\frac{c_1(M)-c_1(W)}{2}\big)\cdot
\big(\eta(\tau)\big)^{3(d-l)}\cdot\prod_{i=1}^d\frac{2\pi\sqrt{-1}x_i}
{\theta(\tau,x_i)}\cdot\prod_{j=1}^l\theta(\tau,w_j-z)\bigg],\end{split}\nonumber\ee
where
$$\eta(\tau):=q^{\frac{1}{24}}\cdot
c=q^{\frac{1}{24}}\prod_{i=1}^{\infty}(1-q^i)$$ is the famous
Dedekind eta function. In particular, $\textrm{Ell}(M,W,\tau,z)$ is
a holomorphic function with respect to the two variables $\tau$ and
$z$ and has no negative powers of $q$ in its Fourier expansion.}}

\end{lemma}

\begin{proof}
\be\begin{split}&\textrm{ch}\big(\textrm{E}(W,q,y)\big)\\
=&c^{2(d-l)}y^{-\frac{l}{2}} \prod_{j=1}^l(1-ye^{-2\pi\sqrt{-1}w_j})
\prod_{n=1}^{\infty}
\frac{\prod_{j=1}^l(1-yq^ne^{-2\pi\sqrt{-1}w_j})
(1-y^{-1}q^ne^{2\pi\sqrt{-1}w_j})}
{\prod_{i=1}^d(1-q^ne^{-2\pi\sqrt{-1}x_i})
(1-q^ne^{2\pi\sqrt{-1}x_i})}\\
=&c^{2(d-l)}y^{-\frac{l}{2}} \prod_{j=1}^l(1-ye^{-2\pi\sqrt{-1}w_j})
\prod_{j=1}^l\frac{\theta(\tau,w_j-z)}{2cq^{\frac{1}{8}}\textrm{sinh}(\pi\sqrt{-1}(w_j-z))}
\prod_{i=1}^d\frac{2cq^{\frac{1}{8}}\textrm{sinh}(\pi\sqrt{-1}x_i)}{\theta(\tau,x_i)}\\
=&\textrm{exp}(\frac{c_1(M)-c_1(W)}{2})\cdot
\big(\eta(\tau)\big)^{3(d-l)}\cdot\prod_{i=1}^d\frac{1-e^{-2\pi\sqrt{-1}x_i}}
{\theta(\tau,x_i)}\cdot\prod_{j=1}^l\theta(\tau,w_j-z)\end{split}\nonumber\ee

The last equality is due to the facts that
$$c_1(M)=\sum_{i=1}^d2\pi\sqrt{-1}x_i\qquad\text{and}\qquad
c_1(W)=\sum_{j=1}^l2\pi\sqrt{-1}w_j.$$

Therefore,
\be\begin{split}&\textrm{Ell}(M,W,\tau,z)\\
=&\int_M\textrm{td}(M)\cdot\textrm{ch}\big(\textrm{E}(W,q,y)\big)\\
=&\int_M\prod_{i=1}^d\frac{2\pi\sqrt{-1}x_i}{1-e^{-2\pi\sqrt{-1}x_i}}
\cdot\textrm{ch}\big(\textrm{E}(W,q,y)\big)\\
=&\int_M\bigg[\textrm{exp}\big(\frac{c_1(M)-c_1(W)}{2}\big)\cdot
\big(\eta(\tau)\big)^{3(d-l)}\cdot\prod_{i=1}^d\frac{2\pi\sqrt{-1}x_i}
{\theta(\tau,x_i)}\cdot\prod_{j=1}^l\theta(\tau,w_j-z)\bigg]\end{split}\nonumber\ee

The holomorphicity of $\textrm{Ell}(M,W,\tau,z)$ is now clear from
this expression as the Jacobi-theta function $\theta(\tau,z)$ only
has zeroes of order $1$ along $z=m_1+m_2\tau$ ($m_1,
m_2\in\mathbb{Z}$) (\cite[p. 59]{Ch}). Also it is obvious from this
expression that $\textrm{Ell}(M,W,\tau,z)$ has no negative powers of
$q$ in its Fourier expansion.
\end{proof}

\emph{Proof of Theorem \ref{theoremofellipticgenus}}.

 $SL_2(\mathbb{Z})$ is generated by the two matrices $$\left(
                           \begin{array}{cc}
                             0 & -1 \\
                             1 & 0 \\
                           \end{array}
                         \right)
                         \qquad\text{and}\qquad
\left(
                           \begin{array}{cc}
                             1 & 1 \\
                             0 & 1 \\
                           \end{array}
                         \right).$$
                        To verify that $\textrm{Ell}(M,W,\tau,z)$ satisfies the required
                        transformation laws, it suffices to
show the following four identities
\be\label{identity1}\textrm{Ell}(M,W,\tau+1,z)=\textrm{Ell}(M,W,\tau,z),\ee
\be\label{identity2}\textrm{Ell}(M,W,\tau,z+1)=(-1)^l\textrm{Ell}(M,W,\tau,z),\ee
\be\label{identity3}\textrm{Ell}(M,W,\tau,z+\tau)=(-1)^l\exp{\big(
-\pi\sqrt{-1}l(\tau+2z)\big)}\textrm{Ell}(M,W,\tau,z),\ee
\be\label{identity4}\textrm{Ell}(M,W,-\frac{1}{\tau},\frac{z}{\tau})
=\tau^{d-l}\exp{(\frac{\pi\sqrt{-1}lz^2}{\tau})}\text{Ell}(M,W,\tau,z).\ee

For Dedekind eta function $\eta(\tau)$ and Jacobi-theta function
$\theta(\tau,z)$ we have the following transformation laws
(\cite{Ch}): \be
\begin{split}&\eta^3(-\frac{1}{\tau})=(\frac{\tau}{\sqrt{-1}})^{\frac{3}{2}}\eta^3(\tau),
\qquad
\eta^3(\tau+1)=e^{\frac{\pi\sqrt{-1}}{4}}\eta^3(\tau),\\
&\theta(\tau,z+1)=-\theta(\tau,z),\qquad \theta(\tau,z+\tau)=
-q^{-\frac{1}{2}}\exp{(-2\pi\sqrt{-1}z)}\theta(\tau,z),\\
&\theta(\tau+1,z)=\exp{(\frac{\pi\sqrt{-1}}{4})}\theta(\tau,z),\qquad
\theta(-\frac{1}{\tau},z)=
 -\sqrt{-1}(\frac{\tau}{\sqrt{-1}})^{\frac{1}{2}}
 \exp{(\pi\sqrt{-1}\tau z^2)}\theta(\tau,\tau z).\end{split}\nonumber\ee

The first three identities (\ref{identity1}), (\ref{identity2}) and
(\ref{identity3}) are easy to verify by using the above-listed
transformation laws. Here we only need to check (\ref{identity4})
carefully. Indeed, \be\label{identity5}\begin{split}
&\prod_{i=1}^d\theta(-\frac{1}{\tau},x_i)\\
=&\prod_{i=1}^d-\sqrt{-1}(\frac{\tau}{\sqrt{-1}})^{\frac{1}{2}}
 \exp{(\pi\sqrt{-1}\tau x_i^2)}\theta(\tau,\tau x_i)\\
=&\exp{(\frac{\tau p_1(M)}{4\pi\sqrt{-1}})}\prod_{i=1}^d
-\sqrt{-1}(\frac{\tau}{\sqrt{-1}})^{\frac{1}{2}}\theta(\tau,\tau
x_i).
\end{split}\ee

Here we use the assumption that
$$p_1(M)=\sum_{i=1}^d(2\pi\sqrt{-1}x_i)^2.$$

Similarly, \be\label{identity6}\begin{split}
&\prod_{j=1}^l\theta(-\frac{1}{\tau},w_i-\frac{z}{\tau})\\
=&\prod_{j=1}^l-\sqrt{-1}(\frac{\tau}{\sqrt{-1}})^{\frac{1}{2}}
 \exp{\big(\pi\sqrt{-1}\tau(w_j-\frac{z}{\tau})^2\big)}\theta(\tau,\tau w_j-z)\\
=&\exp{(\frac{\tau
p_1(W)}{4\pi\sqrt{-1}}+\frac{\pi\sqrt{-1}lz^2}{\tau})}\prod_{j=1}^l
-\sqrt{-1}(\frac{\tau}{\sqrt{-1}})^{\frac{1}{2}}\theta(\tau,\tau
w_j-z).
\end{split}\ee

 In the last equality we use the assumption that $$c_1(W)=\sum_{j=1}^l
 2\pi\sqrt{-1}w_j=0.$$

Combining the transformation law of $\eta(\tau)$, (\ref{identity5}),
(\ref{identity6}) and the fact that $p_1(M)=p_1(W)$ leads to
\be\begin{split} &\textrm{Ell}(M,W,-\frac{1}{\tau},\frac{z}{\tau})\\
=&\int_M\bigg[\textrm{exp}\big(\frac{c_1(M)-c_1(W)}{2}\big)
\big(\eta(-\frac{1}{\tau})\big)^{3(d-l)}\prod_{i=1}^d\frac{2\pi\sqrt{-1}x_i}
{\theta(-\frac{1}{\tau},x_i)}\prod_{j=1}^l\theta(-\frac{1}{\tau},w_j-\frac{z}{\tau})\bigg]\\
=&\tau^{d-l}\exp{(\frac{\pi\sqrt{-1}lz^2}{\tau})}
\int_M\big[\textrm{exp}(\frac{c_1(M)-c_1(W)}{2})
\big(\eta(\tau)\big)^{3(d-l)}\prod_{i=1}^d\frac{2\pi\sqrt{-1}x_i}
{\theta(\tau,\tau x_i)}\prod_{j=1}^l\theta(\tau,\tau w_j-z)\big]\\
=&\tau^{-l}\exp{(\frac{\pi\sqrt{-1}lz^2}{\tau})}
\int_M\big[\textrm{exp}(\frac{c_1(M)-c_1(W)}{2})
\big(\eta(\tau)\big)^{3(d-l)}\prod_{i=1}^d\frac{2\pi\sqrt{-1}(\tau
x_i)}
{\theta(\tau,\tau x_i)}\prod_{j=1}^l\theta(\tau,\tau w_j-z)\big]\\
=&\tau^{d-l}\exp{(\frac{\pi\sqrt{-1}lz^2}{\tau})}\int_M\big[\textrm{exp}\big(\frac{c_1(M)-c_1(W)}{2}\big)
\big(\eta(\tau)\big)^{3(d-l)}\prod_{i=1}^d\frac{2\pi\sqrt{-1}x_i}
{\theta(\tau,x_i)}\prod_{j=1}^l\theta(\tau,w_j-z)\big]\\
=&\tau^{d-l}\exp{(\frac{\pi\sqrt{-1}lz^2}{\tau})}
\textrm{Ell}(M,W,\tau,z)
\end{split}\nonumber\ee

The last but one equality is due to the fact that in the integrand
we are only concerned with the homogeneous part of degree $d$
\big($\textrm{deg}(x_i)=\textrm{deg}(w_j)=1$\big). This completes
the proof of Theorem \ref{theoremofellipticgenus}.

\subsection{Algebraic preliminaries}
Before discussing the arithmetic properties of the generalized
elliptic genus $\textrm{Ell}(M,W,\tau,z)$, we need to review a
well-known manipulation in algebraic number theory of how to derive
modular forms from Jacobi forms.

Recall that the \emph{Eisenstein series} $G_{2k}(\tau)$ are defined
to be (\cite[p. 131]{HBJ})
$$G_{2k}(\tau):=-\frac{B_{2k}}{4k}+\sum_{n=1}^{\infty}\sigma_{2k-1}(n)\cdot q^n,$$
where
$$\sigma_k(n):=\sum_{\text{$m>0$, $m|n$}}
m^k$$ and $B_{2k}$ are the Bernoulli numbers.

These $G_{2k}(\tau)$ carry rich arithmetic information. It is
well-known that $G_{2k}(\tau)$ ($k\geq 2$) are modular forms of
weight $2k$ over the full modular group $SL_2(\mathbb{Z})$ and the
whole graded ring of modular forms over $SL_2(\mathbb{Z})$ are
generated by $G_4(\tau)$ and $G_6(\tau)$. However, $G_{2}(\tau)$ is
\emph{not} a modular form but called \emph{quasi-modular form} as it
transforms as follows (\cite[p. 138]{HBJ}). \be\label{g2}
G_2(\frac{a\tau+b}{c\tau+d})=(c\tau+d)^2
G_2(\tau)-\frac{c(c\tau+d)}{4\pi\sqrt{-1}},\qquad\forall~
 \left(
                           \begin{array}{cc}
                             a & b \\
                             c & d \\
                           \end{array}
                         \right)\in SL_2(\mathbb{Z}).                       \ee

 The following proposition, which is a well-known fact in algebraic number theory and has been used implicitly by
Gritsenko in the proof of \cite[Lemma 1.6]{Gr1}, provides us with a
method for deriving
 modular forms from Jacobi forms.

\begin{proposition}\label{jacobiformprop}
Suppose a function
$\varphi(\tau,z):~\mathbb{H}\times\mathbb{C}\rightarrow\mathbb{C}$
satisfies
\be\label{jacobiformtran}\varphi(\frac{a\tau+b}{c\tau+d},\frac{z}{c\tau+d})
=(c\tau+d)^k\text{exp}(\frac{2\pi\sqrt{-1}mcz^2}{c\tau+d})\cdot\varphi(\tau,z),\qquad\qquad\forall~\left(
                           \begin{array}{cc}
                             a & b \\
                             c & d \\
                           \end{array}
                         \right)
 \in SL_2(\mathbb{Z}).                    \ee
i.e., $\varphi(\tau,z)$ transforms like a Jacobi from of weight $k$
and index $m$.

 Then, if we define
 $$\Phi(\tau,z):=\text{exp}\big(-8\pi^2mG_2(\tau)
 z^2\big)\varphi(\tau,z),$$
we have
\be\label{g22}\Phi(\frac{a\tau+b}{c\tau+d},\frac{z}{c\tau+d})=
(c\tau+d)^k\Phi(\tau,z).\ee

This means, if we set
$$\Phi(\tau,z)=:\sum_{n\in\mathbb{Z}}a_n(\tau)\cdot z^n,$$
then
$$a_n(\frac{a\tau+b}{c\tau+d})=(c\tau+d)^{k+n}a_n(\tau).$$

In particular, if $\varphi(\tau,z)$ is a weak Jacobi form of weight
$k$ and index $m$, then these $a_n(\tau)$ are modular forms of
weight $k+n$ over $SL_2(\mathbb{Z})$.
\end{proposition}

\begin{proof}
(\ref{g22}) can be verified directly by using the assumption
condition (\ref{jacobiformtran}) and the transformation law
(\ref{g2}). If moreover $\varphi(\tau,z)$ is a weak Jacobi form,
then $\varphi(\tau,z)$ and thus $\Phi(\tau,z)$ are holomorphic and
have no negative powers of $q$ when considering their Fourier
expansions in terms of $q$ and $y$. This implies that these
$a_n(\tau)$ are also holomorphic and have no negative powers of $q$
when considering the Fourier expansions of $q$, which gives the
desired proof.
\end{proof}

Now with the assumptions in the Theorem \ref{theoremofellipticgenus}
understood, we know that $\textrm{Ell}(M,W,\tau,z)$ is a weak Jacobi
form of weight $d-l$ and index $\frac{l}{2}$. Then Proposition
\ref{jacobiformprop} tells us that

\begin{proposition}\label{coefficientlemma}

 The series $a_n(M,W,\tau)$ determined by
$$\text{exp}\big[l\cdot G_2(\tau)
\cdot(2\pi\sqrt{-1}z)^2\big]\cdot\textrm{Ell}(M,W,\tau,z)=:
\sum_{n\geq0}a_n(M,W,\tau)\cdot(2\pi\sqrt{-1}z)^n$$ are modular
forms of weight $d-l+n$ over $SL_2(\mathbb{Z})$. Furthermore, the
first three series of $a_n(M,W,\tau)$ are of the following form:
\be\begin{split}
&a_0(M,W,\tau)\\
=&\chi(M,\Lambda_{-1}W^{\ast})+
q\cdot\chi\bigg(M,\Lambda_{-1}W^{\ast}\otimes\big(-2(d-l)-W-W^{\ast}+T+T^{\ast}\big)\bigg)+q^2\cdot(\cdots),\\
&a_1(M,W,\tau)=\sum_{p=0}^l(-1)^p(p-\frac{l}{2})\chi(M,\Lambda^pW^{\ast})+q\cdot(\cdots),\\
&a_2(M,W,\tau)=\big[-\frac{l}{24}\chi(M,\Lambda_{-1}W^{\ast})+\frac{1}{2}\sum_{p=0}^l(-1)^p(p-\frac{l}{2})^2\chi(M,\Lambda^pW^{\ast})
\big]+q\cdot(\cdots).
\end{split}\nonumber\ee

\end{proposition}

\begin{proof}
 The first statement is a direct application of
 Proposition \ref{jacobiformprop} as $\textrm{Ell}(M,W,\tau,z)$ is a weak Jacobi form of weight $d-l$ and index $\frac{l}{2}$.
 For the second one, if we set
$$\textrm{exp}\big[lG_2(\tau)(2\pi\sqrt{-1}z)^2\big]
=:A_0(y)+A_1(y)\cdot q+(\cdots)\cdot q^2,$$ and
$$\textrm{Ell}(M,W,\tau,z)
=:B_0(y)+B_1(y)\cdot q+(\cdots)\cdot q^2,$$ we can easily deduce
from their explicit expressions that \be\begin{split}
A_0(y)&=\textrm{exp}\big[-\frac{l}{24}(2\pi\sqrt{-1}z)^2\big]=1-\frac{l}{24}(2\pi\sqrt{-1}z)^2+\cdots,\\
A_1(y)&=l(2\pi\sqrt{-1}z)^2-\frac{l^2}{24}(2\pi\sqrt{-1}z)^4+\cdots,\\
B_0(y)&=\sum_{p=0}^l(-1)^p\chi(M,\Lambda^pW^{\ast})y^{p-\frac{l}{2}}\\
      &=\sum_{p=0}^l(-1)^p\chi(M,\Lambda^pW^{\ast})[1+(p-\frac{l}{2})(2\pi\sqrt{-1}z)+
         \frac{1}{2}(p-\frac{l}{2})^2(2\pi\sqrt{-1}z)^2+\cdots],\\
B_1(y)&=\chi\bigg(M,\Lambda_{-1}W^{\ast}\otimes
\big(-2(d-l)-W-W^{\ast}+T+T^{\ast}\big)\bigg)+
2\pi\sqrt{-1}z(\cdots).
\end{split}
\nonumber\ee

Note that
\be\begin{split}\sum_{n\geq0}a_n(M,W,\tau)(2\pi\sqrt{-1}z)^n
=A_0(y)B_0(y)+[A_0(y)B_1(y)+A_1(y)B_0(y)]q+\cdots\end{split}\nonumber\ee
then it is easy to deduce the expressions in our
 Proposition \ref{coefficientlemma} in terms of those of $A_0(y)$, $A_1(y)$, $B_0(y)$ and
$B_1(y)$.
\end{proof}

\subsection{$-1$-phenomenon of the generalized elliptic genus}
In this subsection, via Proposition \ref{coefficientlemma} presented
in the last subsection, we will investigate the arithmetic
information of the generalized elliptic genus
$\textrm{Ell}(M,W,\tau,z)$, which can be viewed as an appropriate
$-1$-phenomenon of $\textrm{Ell}(M,W,\tau,z)$.

We will present one proposition and two examples related to
$a_2(M,W,\tau)$, $a_0(M,W,\tau)$ and $a_1(M,W,\tau)$ respectively to
illustrate an appropriate $-1$-phenomenon of the generalized
elliptic genus $\textrm{Ell}(M,W,\tau,z)$.

Our next proposition related to $a_2(M,W,\tau)$ gives the ``reason''
why these $a_n(M,W,\tau)$ should be the $-1$-phenomenon of
$\textrm{Ell}(M,W,\tau,z)$.

\begin{proposition}
$a_2(M,W,\tau)$ is a modular form of weight $d-l+2$ over
$SL_2(\mathbb{Z})$ provided that $p_1(M)=p_1(W)$ and $c_1(W)=0$ in
$H^{\ast}(M,\mathbb{R})$.
 Consequently, if either (i) $d-l$ is odd, or (ii) $d\leq l$ but $d-l\neq-2$, we have
 \be\label{a2}\sum_{p=0}^l(-1)^p(p-\frac{l}{2})^2\chi(M,\Lambda^pW^{\ast})
=\frac{l}{12}\chi(M,\Lambda_{-1}W^{\ast}).\ee

Moreover, if $W=T$ and $c_1(M)=0$ in $H^{\ast}(M,\mathbb{R})$,
(\ref{a2}) is nothing but the original $-1$-phenomenon of the
Hirzebruch $\chi_y$-genus.
\end{proposition}

\begin{proof}
If either (i) $d-l$ is odd or (ii) $d\leq l$ but $d-l\neq-2$,
$a_2(M,W,\tau)$ is a modular form over $SL_2(\mathbb{Z})$ whose
weight is either (i) odd or (ii) no more than $2$ but not zero. This
means $a_2(M,W,\tau)\equiv0$ and then its expression in Proposition
\ref{coefficientlemma} gives (\ref{a2}).

If $W=T$, then
$$\text{the right-hand side of (\ref{a2})}=\frac{d}{12}\chi(M,\Lambda_{-1}T^{\ast})
=\frac{d}{12}\chi_y(M)\big|_{y=-1}=\frac{d}{12}c_d(M).$$

However,
\be\begin{split} &\text{the left-hand side of (\ref{a2})}\\
=&\sum_{p=0}^d(-1)^p(p-\frac{d}{2})^2\chi^p(M)\\
=&\sum_{p=0}^d(-1)^p[2\cdot\frac{p(p-1)}{2}+(1-d)p+\frac{d^2}{4}]\chi^p(M)\\
=&2a_2(M)-(1-d)a_1(M)+\frac{d^2}{4}a_0(M)\\
=&\frac{d(3d-5)}{12}c_d(M)+\frac{(1-d)d}{2}c_d(M)+\frac{d^2}{4}c_d(M)
\qquad\text{\big(via (\ref{expressions}) and $c_1(M)=0$\big)}\\
=&\frac{d}{12}c_d(M)\\
=&\text{the right-hand side of (\ref{a2})}.
\end{split}\nonumber\ee
\end{proof}

The last two examples related to $a_0(M,W,\tau)$ and $a_1(M,W,\tau)$
give much arithmetic information of $M$ and $W$.

\begin{example}
By Proposition \ref{coefficientlemma} we know that $a_0(M,W,\tau)$
is a modular form of weight $d-l$ over $SL_2(\mathbb{Z})$ provided
that $p_1(M)=p_1(W)$ and $c_1(M)=0$ in $H^2(M,\mathbb{R})$.
Consequently,
\begin{enumerate}
\item
if either $d-l$ is odd or $d-l\leq 2$ but is nonzero, we have
$$\chi(M,\Lambda_{-1}W^{\ast})=
\chi\bigg(M,\Lambda_{-1}W^{\ast}\otimes\big(-2(d-l)-W-W^{\ast}+T+T^{\ast}\big)\bigg)=0;$$

\item
if $d-l=4$, $a_0(M,W,\tau)$ is proportional to the Eisenstein series
$$G_4(\tau)=-\frac{B_{4}}{8}+q+\cdots=\frac{1}{240}+q+\cdots$$

and so
$$\chi\bigg(M,\Lambda_{-1}W^{\ast}\otimes\big(-2(d-l)-W-W^{\ast}+T+T^{\ast}\big)\bigg)=
240\chi(M,\Lambda_{-1}W^{\ast});$$

\item
if $d-l=6$, $a_0(M,W,\tau)$ is proportional to the Eisenstein series
$$G_6(\tau)=-\frac{B_{6}}{12}+q+\cdots=-\frac{1}{504}+q+\cdots$$

and so
$$\chi\bigg(M,\Lambda_{-1}W^{\ast}\otimes\big(-2(d-l)-W-W^{\ast}+T+T^{\ast}\big)\bigg)=
-504\chi(M,\Lambda_{-1}W^{\ast});$$

\item
if $d-l=8$, $a_0(M,W,\tau)$ is proportional to
$$[G_4(\tau)]^2=[\frac{1}{240}+q+\cdots]^2=\frac{1}{240^2}
+\frac{1}{120}q+\cdots$$

and so
$$\chi\bigg(M,\Lambda_{-1}W^{\ast}\otimes\big(-2(d-l)-W-W^{\ast}+T+T^{\ast}\big)\bigg)=
480\chi(M,\Lambda_{-1}W^{\ast}).$$
\end{enumerate}
\end{example}

\begin{example}
By Proposition \ref{coefficientlemma} we know that $a_1(M,W,\tau)$
is a modular form of weight $d-l+1$ over $SL_2(\mathbb{Z})$ provided
that $p_1(M)=p_1(W)$ and $c_1(M)=0$ in $ H^2(M,\mathbb{R})$.
Consequently, if either $d-l$ is even or $d-l\leq 1$ but $d-l\neq
-1$, we have
$$\sum_{p=0}^l(-1)^p(p-\frac{l}{2})\chi(M,\Lambda^pW^{\ast})=0.$$
\end{example}

\bibliographystyle{amsplain}

\end{document}